\documentclass[12pt]{amsart}
\usepackage{amsmath}
\usepackage{amssymb}
\usepackage{amsfonts}
\usepackage{latexsym}
\usepackage{amscd}
\usepackage[mathscr]{euscript}
\usepackage{enumitem}
\usepackage{xy} \xyoption{all}
\usepackage{pdfsync}
\usepackage[active]{srcltx}

\newcommand{\N}{{\mathbb{N}}}

\newcommand{\Z}{{\mathbb{Z}}}

\newcommand{\C}{{\mathbb{C}}}

\newcommand{\uloopr}[1]{\ar@'{@+{[0,0]+(-4,5)}@+{[0,0]+(0,10)}@+{[0,0] +(4,5)}}^{#1}}
\newcommand{\uloopd}[1]{\ar@'{@+{[0,0]+(5,4)}@+{[0,0]+(10,0)}@+{[0,0]+ (5,-4)}}^{#1}}
\newcommand{\dloopr}[1]{\ar@'{@+{[0,0]+(-4,-5)}@+{[0,0]+(0,-10)}@+{[0, 0]+(4,-5)}}_{#1}}
\newcommand{\dloopd}[1]{\ar@'{@+{[0,0]+(-5,4)}@+{[0,0]+(-10,0)}@+{[0,0 ]+(-5,-4)}}_{#1}}

\newcommand{\luloop}[1]{\ar@'{@+{[0,0]+(-8,2)}@+{[0,0]+(-10,10)}@+{[0, 0]+(2,2)}}^{#1}}

\newtheorem{lem}{Lemma}[section]
\newtheorem{corol}[lem]{Corollary}
\newtheorem{theor}[lem]{Theorem}
\newtheorem{prop}[lem]{Proposition}
\theoremstyle{definition}
\newtheorem{defi}[lem]{Definition}

\newtheorem{exem}[lem]{Example}

\newtheorem{rema}[lem]{Remark}

\begin{document}
\title[The structure of  crossed products by endomorphisms]{The structure of  crossed products by endomorphisms}%
\author{Eduard Ortega}
\address{Department of Mathematical Sciences\\
NTNU\\
NO-7491 Trondheim\\
Norway } \email{eduardor@math.ntnu.no}

\author{Enrique Pardo}
\address{Departamento de Matem\'aticas, Facultad de Ciencias\\ Universidad de C\'adiz, Campus de
Puerto Real\\ 11510 Puerto Real (C\'adiz)\\ Spain.}
\email{enrique.pardo@uca.es}\urladdr{https://sites.google.com/a/gm.uca.es/enrique-pardo-s-home-page/}

\thanks{This research was supported by the NordForsk Research Network
  ``Operator Algebras and Dynamics'' (grant \#11580). The first author was partly supported by  MEC-DGESIC (Spain) through Project MTM2008-06201-C02-01/MTM. The second author was partially supported by the DGI and European Regional Development Fund, jointly, through Project MTM2008-06201-C02-02 and by PAI III grants FQM-298 and P07-FQM-02467 of the Junta de Andaluc\'{\i}a. Both authors were partially supported by the DGI and European Regional Development Fund, jointly, through Project MTM2011-28992-C02-02, by the Consolider Ingenio
``Mathematica" project CSD2006-32 by the MEC and by 2009 SGR 1389 grant of the Comissionat per Universitats i Recerca de la Generalitat de
Catalunya.} \subjclass[2000]{Primary 16D70, 46L35; Secondary
06A12, 06F05, 46L80} \keywords{}
\date{\today}
\begin{abstract}
We describe simplicity of the Stacey crossed product $A\times_\beta \N$ in terms of conditions of the endomorphism $\beta$. 
Then, we use a characterization of the graph $C^*$-algebras $C^*(E)$ as the Stacey crossed product  $C^*(E)^\gamma \times_{\beta_E} \N$ to study its ideal properties, in terms of the (non-classical) $C^*$-dynamical system  $(C^*(E)^\gamma, \beta_E )$.
Finally, we give sufficient conditions for the Stacey crossed product $A\times_\beta \N$ being a purely infinite simple $C^*$-algebra.
\end{abstract}

\maketitle

In \cite{Cu}, Cuntz defined the fundamental Cuntz algebras $\mathcal{O}_n$. He also represented these algebras as crossed products of a UHF-algebra by an endomorphism, and he used this representation to prove the simplicity of his algebras.  In a subsequent paper \cite{Cu2} he saw this construction as a full corner of an ordinary crossed product. However Cuntz did not explain what kind of crossed product
by an endomorphism was. Later, Paschke \cite{Pa} gave an elegant generalization of Cuntz's result, and described the crossed product of a unital $C^*$-algebra by an endomorphism $\beta:A\rightarrow A$, written $A\times_\beta \N$, as the $C^*$-algebra generated by $A$ and an isometry $V$, such that $VaV^*=\beta(a)$. Endomorphisms of $C^*$-algebras appeared elsewhere
(cf. \cite{Arw}, \cite{Dop} and the references given there), and this led Stacey to give a modern description of their crossed products in terms of covariant representations and universal properties \cite{Stacey}. He also verified that the candidate proposed in \cite{Cu}
had the required property. See \cite{Ad} and \cite{Bo} for further study and generalization of the Stacey's crossed product.

Cuntz's representation of the $\mathcal{O}_n$ as crossed products by an endomorphism aimed to prove the simplicity of these $C^*$-algebras. Paschke gave conditions on the $C^*$-algebra $A$ and in the isometry to obtain a simple crossed product \cite[Proposition 2.1]{Pa}, later improved in \cite[Corollary 2.6]{Bo}. But it is in \cite[Theorem 4.1]{S} where the most powerful result about the simplicity of the Stacey crossed product is given. Namely, If $A$ is a unital $C^*$-algebra and $\beta$ is an injective $*$-endomorphism, then $A\times_\beta \N$ is simple and $\beta(1)$ is a full projection in $A$ if and only if $\beta^n$ is outer for every $n>0$  and there are no non-trivial ideals $I$ of $A$ with $\beta(I)\subseteq I$. Schweizer used the representation of the Stacey crossed product as Cuntz Pimsner algebra given by Muhly and Solel \cite{MS}. 

The theory of graph $C^*$-algebras $C^*(E)$ has been developed by a number of researchers (see \cite{Bates}, \cite{Bates2} and \cite{Rae}, among others) in an attempt to produce a far-reaching and yet accessible generalization of the Cuntz-Krieger algebras of finite matrices. Indeed, graph algebras do provide a large and interesting class of examples of $C^*$-algebras, both simple and non-simple ones. For example, Cuntz's algebras are $C^*$-algebras of a graph.

In \cite{AR} an Huef and Raeburn study the crossed products of an Exel system, and they prove that the relative Cuntz-Pimsner algebra of an Exel system is isomorphic to a Stacey crossed product of its core. This result leads them to a realization of the graph algebra $C^*(E)$ as a Stacey crossed product $C^*(E)^\gamma\times_{\beta_E}\N$ by an endomorphism of the core, extending the work of Kwa\'{s}niewski on finite graphs \cite{Kw}. 

In the case of Leavitt path algebras (see e.g. \cite{Ab}), this result appears in more simple form in \cite[Section 2]{Ara}, where the authors give a representation of Leavitt path algebras of a finite graph without sinks and sources as a fractional skew monoid rings (the algebraic analog of the crossed product by an endomorphism).    

The aim of this paper is to study the simplicity of the non-unital crossed product. Our fundamental technique is seeing the Stacey crossed product $A\times_\beta \N$ as a full corner of a crossed product by an automorphism $P( A_\infty\times_{\beta_\infty}\Z)P$  (see \cite{Cu2,Stacey}), where $P$ is a full projection of the multipliers that is invariant under the canonical gauge-action. Therefore, we can define the associated \emph{Connes' Spectrum} of the endomorphism in a similar way we do it for an automorphism (see \cite{O,OP1,JKO}) and construct a parallel Connes' spectrum theory for endomorphisms.  Hence,  following the results of Olesen and Pedersen \cite{OP1,OP3} we characterize simplicity for the Stacey crossed product $A\times_\beta \N$.

As an example, we use the characterization of graph $C^*$-algebras $C^*(E)$ as Stacey crossed product ${C^*(E)}^{\gamma}\times_{\beta_E}\N$ \cite{AR}, where in this case $C^*(E)^\gamma$ the core, that is a (non-unital) AF-algebra, and $\beta_E$ is a corner isomorphism. However, although the characterization of the simplicity of $C^*(E)$ is well understood in terms of properties of the graph \cite{Bates}, our intention is to describe this characterization in terms of the non-classical $C^*$-dynamical system $(C^*(E)^\gamma,\beta_E)$.  We also give conditions on the $C^*$-dynamical system to satisfy the Cuntz-Krieger uniqueness theorem: for any faithful covariant representation $(\pi,V)$ of $(C^*(E)^\gamma,\beta_E)$ we have $C^*(\pi,V)\cong C^*(E)$. Finally, by using ideas from \cite{Ror,JKO}, we give sufficient conditions on $A$ and the endomorphism $\beta$ in order to guarantee that $A\times_\beta \N$ is simple and purely infinite. The main difference between these previous results and ours is that we do not ask the $C^*$-algebra $A$ to be simple.

The contents of this paper can be summarized as follows: In Section $1$ we give the basic definitions of a Stacey crossed product. We use the characterization of the Stacey crossed product as a Cuntz Pimsner algebra \cite{MS} to describe the gauge invariant ideals, using a result from Katsura \cite{Ka2}. Then we define the  Connes' spectrum of an endomorphism \cite{O}, a technical device that, with the help of results from Olsen and Pedersen (\cite{OP1,OP3}), allows us to give necessary and sufficient conditions to state the simplicity of a Stacey crossed product. In Section $2$ we apply our results to graph $C^*$-algebras. We recall the definition of the graph endomorphism $\beta_E:C^*(E)^\gamma\rightarrow C^*(E)^\gamma$ of the core of the graph $C^*$-algebra \cite[Theorem  9.3]{AR}, used to prove that $C^*(E)\cong C^*(E)^\gamma\times_{\beta_E} \N$. 
Then, we characterize condition (L) of the graph $E$ in terms of the endomorphism $\beta_E$: every cycle has an entry. Condition (L) in $E$ implies that $C^*(E)$ satisfies the Cuntz-Krieger uniqueness theorem (see e.g. \cite[Section 2]{Rae}). Thus, we use our previous results to give the (well-known) necessary and sufficient condition of the graph $E$ for the graph $C^*$-algebra $C^*(E)$ being simple. Finally, in Section $3$, we give  sufficient conditions on the $C^*$-algebra $A$ and the endomorphism $\beta:A\longrightarrow A$ for $A\times_\beta \N$ being a unital simple and purely infinite $C^*$-algebra.

\section{Simple Stacey crossed product}
The pair $(A,\beta)$, where $A$ is a $C^*$-algebra and $\beta:A\rightarrow A$ an injective endomorphism, is called a \emph{$C^*$-dynamical system}. 

\begin{defi} We say that $(\pi,V)$ is a \emph{Stacey covariant representation} of $(A,\beta)$ if $\pi:A\rightarrow \mathcal{B}(\mathcal{H})$ is a non-degenerated  representation and $V$ is an isometry of $\mathcal{B}(\mathcal{H})$ such that $\pi(\beta(a))=V\pi(a)V^*$ for every $a\in A$.  We say that $(\pi,V)$ is \emph{faithful} if $\pi$ is faithful, and we denote by $C^*(\pi,V)$ the $C^*$-algebra generated by $\{\pi(A)V^n(V^m)^*\}_{n,m\geq 0}$.  
\end{defi}

Stacey showed in \cite{Stacey} that there exists a $C^*$-algebra that is generated by a universal Stacey covariant representation $(\iota_\infty, V_\infty)$. We call  $A\times_\beta  \N:=C^*(\iota_\infty, V_{\infty})$ the \emph{Stacey crossed product} of $A$ by the endomorphism $\beta$.

\begin{rema} Observe that, if $\beta$ is an automorphism, then $A\times_\beta  \N$ is the usual crossed product $A\times_\beta  \Z$.
\end{rema}

Given $z\in \mathbb{T}$, we define an automorphism in $A\times_\beta  \N$ by the rule $\gamma_z(a)=a$ and $\gamma_z(V_\infty) =zV_\infty$ for every $a\in A$. It defines the gauge action $\gamma:\mathbb{T}\rightarrow \text{Aut} (A\times_\beta  \N)$. An ideal $I$ of $A\times_\beta  \N$ is said to be \emph{gauge invariant} if $\gamma_z(I)=I$ for every $z\in \mathbb{T}$. We define a canonical faithful conditional expectation $E:A\times_\beta  \N\longrightarrow A$ as $E(x):=\int_\mathbb{T} \gamma_z(x) dz$ for every $x\in  A\times_\beta  \N$.

We say that the  endomorphism $\beta:A\longrightarrow A$ is \emph{extendible} if, given any strong convergent sequence  $\{x_n\}_{n\geq 0}\subset A$, then the sequence $\{\beta(x_n)\}_{n\geq 0}$ converges in the strong topology (i.e., $\beta$ extends to   $\widehat{\beta}:M(A)\longrightarrow M(A)$).  Observe that, if $\beta$ is injective, then $\widehat{\beta}(a)\in A$ implies that $a\in A$. Indeed, let $\{a_n\}$ be a sequence that converges in the strong topology and such that $\{\beta(a_n)\}$ converges in norm topology. Since $\beta$ is isometric ($\beta$ is injective) then $\{a_n\}$ converges in the norm topology too. 

We define the inductive system $\{A_i,\gamma_i\}_{i\geq 0}$ given by $A_i:=A$ and $\gamma_i=\beta$ for every $i\geq 0$. Let $A_\infty:=\varinjlim  \{A_i,\gamma_i\}$. For any $i\geq 0$, $\varphi_i:A_i\longrightarrow A_\infty$ denotes the (injective) canonical map.  The diagram 
\[
{
\def\labelstyle{\displaystyle}
 \xymatrix{ A \ar[d]^{\beta}\ar[r]^{\beta}  & A\ar[d]^{\beta}\ar[r]^{\beta} & A\ar[d]^{\beta}\ar[r]^\beta & \cdots  \\ A \ar[r]_{\beta}  & A\ar[r]_{\beta} & A\ar[r]_\beta & \cdots  }
}\]
gives rise to an automorphism $\beta_\infty:A_\infty\longrightarrow A_\infty$. 

Observe that, if $\beta$ is  an extendible endomorphism, then $\varphi_0$ extends to $\widehat{\varphi_0}:M(A)\longrightarrow M(A_\infty)$. 

\begin{prop}[{cf. \cite[Proposition 3.3]{S}}]\label{Morita_eq} If $A$ is a $C^*$-algebra and $\beta:A\longrightarrow A$ is an extendible and injective endomorphism, then  $A\times_\beta \N\cong P(A_\infty \times_{\beta_\infty} \Z)P$, where $P=\widehat{\varphi_0}(1_{M(A)})\in M(A_\infty \times_{\beta_\infty} \Z)$. Moreover, $P$ is a full projection, so that $A\times_\beta \N$ is strongly Morita equivalent to $A_\infty \times_{\beta_\infty} \Z$.
\end{prop}

Therefore, there there exist a bijection between the ideals of $A_\infty \times_{\beta_\infty} \Z$ and $A\times_\beta \N$ given by $$I\longmapsto PIP\qquad\text{and}\qquad J\longmapsto \overline{(A_\infty \times_{\beta_\infty} \Z )J (A_\infty \times_{\beta_\infty} \Z) }\,.$$ Moreover, if $U_\infty $ is the unitary that implements the automorphism $\beta_\infty$, then $V_\infty = PU_\infty P$ is the isometry implementing $\beta$. Since $\gamma_z(P)=P$ for every $z\in \mathbb{T}$, the canonical gauge action $\gamma:\mathbb{T}\longrightarrow \text{ Aut }(A_\infty \times_{\beta_\infty} \Z)$ restricts to the gauge action of $A\times_\beta \N$. 

\begin{lem} If $A$ is a $C^*$-algebra and $\beta:A\longrightarrow A$ is an extendible and injective endomorphism, then there exists an order preserving  bijection between gauge invariant ideals of $A\times_\beta \N$ and $A_\infty \times_{\beta_\infty} \Z$.
\end{lem}

Now, we will describe the gauge invariant ideals in terms of the $C^*$-dynamical system $(A,\beta)$.
Given an endomorphism $\beta:A\longrightarrow A$, it is easy to check that $\beta(A)$ is a hereditary sub-$C^*$-algebra of $A$ if and only if $\overline{\beta(A)A\beta(A)}=\beta(A)$.

\begin{defi}
Let $A$ be a $C^*$-algebra and let  $\beta:A\rightarrow A$ an endomorphism such that $\beta(A)$ is a hereditary sub-$C^*$-algebra of $A$. We say that an ideal $I$ of $A$ is \emph{$\beta$-invariant} if  $\overline{\beta(A)I\beta(A)}=\beta(I)$. We say that $A$ is \emph{$\beta$-simple} if there are no non-trivial $\beta$-invariant ideals.
\end{defi}

\begin{lem}\label{invariant} 
Let $A$ be a $C^*$-algebra, and let  $\beta:A\rightarrow A$ an endomorphism such that $\beta(A)$ is a hereditary sub-$C^*$-algebra of $A$. If $I$ is a $\beta$-invariant ideal of $A$, then it is also $\beta^n$-invariant for every $n>0$. 
\end{lem}
\begin{proof} Let $I$ be an ideal such that $\overline{\beta(A)I\beta(A)}=\beta(I)$. We will prove the result by induction on $n$. The case $n=1$ being clear, suppose that $\overline{\beta^{n-1}(A)I\beta^{n-1}(A)}=\beta^{n-1}(I)$.  Observe that, since $\beta(A)$ is a hereditary sub-$C^*$-algebra of $A$, we have that $\beta(A)=\overline{\beta(A)A\beta(A)}$. Thus,
\begin{align*} \overline{\beta^n(A)I\beta^n(A)} & = \overline{\beta^{n-1}(\beta(A)A\beta(A))I\beta^{n-1}(\beta(A)A\beta(A))} \\
 & \subseteq \overline{\beta^{n}(A)\beta^{n-1}(A) I \beta^{n-1}(A) \beta^n(A)} \\
 & = \overline{\beta^n(A)\beta^{n-1}(I)\beta^n(A)}=\beta^{n-1}(\overline{\beta(A)I\beta(A)})=\beta^n(I)\,.
\end{align*}
Therefore, $\overline{\beta^n(A)I\beta^n(A)}=\beta^n(I)$ as desired.
\end{proof}

\begin{rema}\label{rema_inv} 
Notice that the converse of the above Lemma is not true in general. Let $A= C_0(\Z)$ and let ${\beta}:C_0(\Z)\rightarrow C_0(\Z)$ be the automorphism  that sends $\chi_{\{i\}}$ (the characteristic function at $i$) to $\chi_{\{i+1\}}$ for every $i\in\Z$.  It is clear that $C_0(\Z)$ is $\beta$-simple, but $I=C_0(2\cdot \Z)$ is a $\beta^2$-invariant ideal.
\end{rema}

Observe also that, if $I$ is a $\beta$-invariant ideal, then $\beta(I)$ is a hereditary sub-$C^*$-algebra of $A$, but the above example also shows that the converse it is not true.

\begin{rema}\label{corner} Let $\beta$ be  an injective and extendible endomorphism such that $\beta(A)$ is hereditary. 
If we set the projection $P=\widehat{\varphi_0}(1_{M(A)})=(1, P_1, P_2,\ldots)\in M(A_\infty)$, where $P_n=\widehat{\beta}^n(1_{M(A)})$, then we have that $A\cong \varphi_0(A)=PA_\infty P$. Hence, we can see $A$ as a hereditary sub-$C^*$-algebra of $A_\infty$ such that ${\beta_\infty}_{|A}=\beta$. Indeed, it is enough to check that, given any $n\in \N$ and $a\in A$, then $P\varphi_n(a)P=\widehat{\varphi_n}(P_naP_n)\in \varphi_0(A)$. But since $P_naP_n\in \overline{\beta^n(A)A\beta^n(A)}=\beta^n(A)$ (by Lemma \ref{invariant}), we have that $P\varphi_n(a)P\in \varphi_n(\beta^n(A))=\varphi_0(A)$.
\end{rema}

In \cite{Pi} Pimsner introduced a  class of $C^*$-algebras (later improved by Katsura \cite{Ka1}) generated by $C^*$-correspondences ($X$, $\varphi_X$) over $A$, called Cuntz-Pimsner algebras and denoted by $\mathcal{O}_X$. In particular this class includes crossed products and graph $C^*$-algebras. Katsura \cite{Ka2} studies gauge-invariant ideals of Cuntz-Pimsner algebras; in particular, when $X$ is a Hilbert $A$-bimodule (see e.g. \cite{AEE}), he obtain a bijection between gauge invariant ideals of the Cuntz-Pimsner algebra $\mathcal{O}_X$ and invariant ideals $I$ of $A$ with respect to the correspondence $X$ (i.e., $\varphi_X(I)X=XI$) \cite[Theorem 10.6]{Ka2}. 

Let $\beta:A\longrightarrow A$ is an injective endomorphism such that $\beta(A)$ is a hereditary sub-$C^*$-algebra. If we set $X:={}_{\beta}A=\beta(A)A$ with left-action $\varphi_X$ given by the  endomorphism $\beta$, and right inner product given by $<x,y>_A=x^*y$ for every $x,y\in A$, then we have a $C^*$-correspondence. We have that $\varphi_X(A)\subseteq \mathcal{K}(X)$ (the compact operators of $X$), and since $\beta(A)$ is a hereditary sub-$C^*$-algebra, it follows that  $\overline{\beta(A)A\beta(A)}=\beta(A)$, whence $\varphi_X(A)= \mathcal{K}(X)$. Therefore, since $\beta$ is injective and  $\varphi_X(A)=\mathcal{K}(X)$, we can define a left inner product as ${}_A<x,y>:=\varphi_{X}^{-1}(\theta_{x,y})$ for every $x,y\in A$. Hence, $X$ has a natural structure of Hilbert $A$-bimodule. 

\begin{lem}[{cf. \cite{MS}}]\label{otroketal}  
If $A$ is a $C^*$-algebra, $\beta:A\longrightarrow A$ is an injective endomorphism such that $\beta(A)$ is a hereditary sub-$C^*$-algebra of $A$ and $X={}_\beta A$ is the Hilbert $A$-bimodule defined above, then $\mathcal{O}_X\cong A\times_{\beta}\N$. 
\end{lem}

Thus, we can apply Katsura's description of the gauge invariant ideals, and we see that an ideal $I$ of $A$ is invariant with respect to the correspondence $X$ if and only if $\beta(A)I=\beta(I)A$. 

\begin{lem}\label{inv_ideal} 
If $A$ is a $C^*$-algebra and $\beta:A\longrightarrow A$ is an injective endomorphism such that $\beta(A)$ is a hereditary sub-$C^*$-algebra of $A$, then $I$ is a $\beta$-invariant ideal of $A$ if and only if $\beta(A)I=\beta(I)A$.
\end{lem}
\begin{proof} First, suppose that $\beta(A)I=\beta(I)A$, and observe that $\beta(I)\subseteq I$. Thus we have that 
$$\beta(I)A\beta(I)=\beta(A)I\beta(A)\subseteq \overline{\beta(A)A\beta(A)}=\beta(A)\,.$$
Then multiplying at both sides by $\beta(I)$ we have $\beta(A)I\beta(A)=\beta(I)A\beta(I)\subseteq\beta(I)$, and therefore $\overline{\beta(A)I\beta(A)}=\beta(I)$.

In the other side, suppose that $\overline{\beta(A)I\beta(A)}=\beta(I)$. From $\beta(I)\subseteq I$ it follows $\beta(I)A\subseteq\beta(A)I$.  Now, let $\{e_n\}\subset I_+$ be an approximate unit of $I$, and let $a\in A$ and $y\in I$. We claim that $\beta(e_n)\beta(a)y=\beta(e_na)y$ converges to $\beta(a)y$, whence $\beta(a)y\in \beta(I)A$. Indeed, let $z\in I$ such that $\beta(a)yy^*\beta(a^*)=\beta(z)$. Given $\varepsilon>0$ there exists $n\in \N$ such that $\|e_nz-z\|<\varepsilon/2$. Then we have
\begin{align*}  \|\beta(e_na)y-\beta(a)y\|^2 & = \|(\beta(e_na)y-\beta(a)y)(\beta(e_na)y-\beta(a)y)^*\| \\
& \leq \|  \beta(e_na)yy^*\beta(a^*e_n)-\beta(a)yy^*\beta(a^*e_n)\| + \| \beta(e_na)yy^*\beta(a^*)-\beta(a)yy^*\beta(a^*)\| \\
& \leq \|  \beta(e_nze_n)-\beta(ze_n)\| + \| \beta(e_nz)-\beta(z)\| \\
& =\|  e_nze_n-ze_n\| + \| e_nz-z\| < \varepsilon/2+\varepsilon/2=\varepsilon \,.
\end{align*}

Thus $\beta(e_na)y$ converges to $\beta(a)y$, as desired.
\end{proof}

\begin{prop}\label{gauge_inv} 
If $A$ is a $C^*$-algebra and $\beta:A\longrightarrow A$ is an injective endomorphism such that $\beta(A)$ is a hereditary sub-$C^*$-algebra of $A$, then there is a bijection between gauge invariant ideals of $A\times_\beta \N$ and $\beta$-invariant ideals of $A$. Thus, $A$ is $\beta$-simple if and only if $A_\infty$  is $\beta_\infty$-simple.
\end{prop}
\begin{proof}
First statement holds from  \cite[Theorem 10.6]{Ka2} and Lemmas  \ref{otroketal} \& \ref{inv_ideal}. Last statement follows from \cite[Lemma 6.1]{OP1}.
\end{proof}

\begin{rema}\label{desprescaldra}
The bijection stated in Proposition \ref{gauge_inv}  sends $I\mapsto \overline{(A\times_\beta \N)\cdot I \cdot(A\times_\beta \N)}$ and $K\mapsto K\cap A.$
\end{rema}

Finally we give necessary and sufficient conditions for the simplicity of a Stacey crossed product. The main technical device we use is the Connes' spectrum of an endomorphism. This is just a reformulation of the Connes' spectrum for automorphisms (see \cite{O,JKO}). We will see that for nice endomorphisms (extendible and hereditary image) the  Connes' spectrum of $\beta$ and that of the associated automorphism $\beta_\infty$ coincide.  Therefore, we will be able to use results by Olesen and Pedersen  to determine the conditions for the simplicity of the Stacey crossed products.
 
\begin{defi} 
Let $A$ be a $C^*$-algebra and let $\beta:A\rightarrow A$ be an endomorphism. Then we say that:
\begin{enumerate}
\item $\beta$ is \emph{inner} if there exists an isometry $W\in M(A)$ such that $\beta=\text{Ad }W$.
\item $\beta$ is \emph{outer} if it is not inner.
\item $\beta$ is \emph{weakly properly outer} if for every $\beta$-invariant ideal $I$ and every $n>0$ the restriction endomorphism $\beta^n_{|I}$ is outer.
\end{enumerate}
\end{defi}

Recall \cite[Definition 2.1]{El} that an automorphism $\alpha$ of a $C^*$-algebra $I$ is said to be properly outer if for every nonzero $\alpha$-invariant two-sided ideal $I$ of $A$ and for every unitary multiplier $u$ of $I$, $\Vert \alpha \vert I-\mbox{Ad}_u\vert I\Vert$. When $\beta$ is an automorphism, the notion of weakly properly outerness is weaker than the properly outer notion by Elliott \cite{El}, later studied by Olesen and Pedersen \cite[Theorem 10.4]{OP3}. 

\begin{defi}  
Let $A$ be a $C^*$-algebra, let $\beta:A\rightarrow A$ be an extendible injective endomorphism and let $\gamma:\mathbb{T}\longrightarrow \text{Aut }(A\times_\beta \N)$ be the gauge action. We define the \emph{Connes' spectrum} of $\beta$ as 
$$\mathbb{T}(\beta):=\{t\in\mathbb{T}: \gamma_t(I)\cap I\neq 0 \text{ for every }0\neq I\lhd A\times_\beta \N\}\,.$$
\end{defi}

\begin{rema}
Observe that $\mathbb{T}(\beta)$ is a closed subgroup of $\mathbb{T}$. Hence can only be $\{ 1\}$, $\mathbb{T}$ or a finite subgroup.
\end{rema}

This definition of the Connes' spectrum coincide with the one given by Olesen and Olesen \& Pedersen \cite{O,OP1} when $\beta$ is an automorphism.  Moreover, using that the bijection between ideals of $A\times_\beta \N$ and these of $A_\infty \times_{\beta_\infty}\Z$ given by
$$I\longmapsto PIP\qquad\text{and}\qquad J\longmapsto \overline{(A_\infty \times_{\beta_\infty} \Z )J (A_\infty \times_{\beta_\infty} \Z) }\,,$$ 
and the fact that the canonical gauge action $\gamma:\mathbb{T}\longrightarrow \text{ Aut }(A_\infty \times_{\beta_\infty} \Z)$ restricts to the gauge action of $A\times_\beta \N$ (since $\gamma_z(P)=P$ for every $z\in \mathbb{T}$), the following lemma easily follows.

\begin{lem}
If $A$ is a $C^*$-algebra and $\beta:A\rightarrow A$ is an extendible injective endomorphism with $\beta(A)$ being a hereditary sub-$C^*$-algebra of $A$, then $\mathbb{T}(\beta)=\mathbb{T}(\beta_\infty)$. 
\end{lem}

Let $\beta$ be an extendible endomorphism such that $\beta=\text{Ad } V$, where $V$ is an isometry of $M(A)$. Then we can construct a unitary of $M(A_\infty)$ $U:=\sum_{i\geq 0}\widehat{\varphi_i}(V)$ such that $\beta_\infty=\text{Ad }U$. Now, let us see a result following from \cite{O}.

\begin{theor}\label{theor_simple} 
Let $A$ be a $C^*$-algebra and let $\beta:A\rightarrow A$ be an extendible injective endomorphism with $\beta(A)$ being a hereditary sub-$C^*$-algebra of $A$. Let us consider the following statements:
\begin{enumerate}
\item $\mathbb{T}(\beta^n)=\mathbb{T}$ for every $n>0$.
\item Given $a\in A\mbox{ }^{\widetilde{ }}$ (the unitization of $A$) and any $B$ hereditary sub-$C^*$-algebra of $A$, for every $n>0$ we have that
$$\text{inf }\{\|xa\beta^n(x)\|: 0\leq x\in B \text{ with }\|x\|=1\}=0\,.$$
\item $\beta^n$ is outer for every $n>0$.
\end{enumerate}
Then, $(1)\Rightarrow (2) \Rightarrow (3)$. Moreover, if $A$ is $\beta$-simple, then $(3)\Rightarrow (1)$ (and thus all they are equivalent).
\end{theor}
\begin{proof} 

$(1)\Rightarrow (2)$ This is \cite[Theorem 10.4 and Lemma 7.1]{OP3}. If $\mathbb{T}(\beta^n)=\mathbb{T}$ then $\mathbb{T}(\beta^n_\infty)=\mathbb{T}$ for every $n>0$, so $\beta^n_\infty$ is properly outer for every $n>0$. Since any  hereditary sub-$C^*$-algebra $B$ of $A$ is also a hereditary sub-$C^*$-algebra of $A_\infty$, (see Remark \ref{corner}), we can apply \cite[Proof of Lemma 7.1]{OP3} to $B$. Thus, since ${\beta^n_\infty}_{|A}=\beta^n$, we have the result.

$(2)\Rightarrow (3)$ Suppose that $\beta^n=\text{Ad }W$ for an isometry $W\in M(A)$. Fix $\varepsilon>0$, and take $b\in A_+$ with $\|b\|=1$. Set $c:=f_\varepsilon(b)$, where $f_\varepsilon(t):[0,1]\longrightarrow \mathbb{R}_+$ is the continuous function that is $f_\varepsilon(0)=0$, constant $1$ for $t\geq \varepsilon$ and linear otherwise. Then, we have that $xc=cx=x$ for every $x\in \overline{(b-\varepsilon)_+A(b-\varepsilon)_+}$. Hence, given any $0\leq x\in \overline{(b-\varepsilon)_+A(b-\varepsilon)_+}$ with $\|x\|=1$, we have that
\begin{align*} \|x(cW^*)\beta^n(x)\|^2 & = \| x(cW^*)WxW^*\|^2 =\| xcxW^*\|^2\\
 & =\|x^2W^*\|^2= \|x^2W^*Wx^2\|=\|x^4\|=\|x\|^4=1\,,
\end{align*} 
which contradicts the hypothesis, since $cW^*\in A$.

Now, suppose that $A$ is $\beta$-simple. We are going to prove 
that $(3)\Rightarrow (1)$. By \cite[Theorem 10.4]{OP3} we have that $\mathbb{T}(\beta_\infty)=\mathbb{T}$ if and only if $\mathbb{T}(\beta^n_\infty)=\mathbb{T}$ for every $n\in\N$. Let us suppose that $\mathbb{T}(\beta)=\mathbb{T}(\beta_\infty)\neq\mathbb{T}$. Hence, $\mathbb{T}(\beta_\infty)$ is a finite subgroup, and thus the complement $\mathbb{T}(\beta_\infty)^{\bot}\neq \{0\}$.  Therefore, by \cite[Theorem 4.5]{OP3}, for every $k\in \mathbb{T}(\beta_\infty)^{\bot}$ we have that $\beta_\infty^k=\text{Ad }U$, where $U\in M(A_\infty)$. But then $V=PUP\in M(A)$ is an isometry such that $\beta^k=\text{Ad }V$, a contradiction.
\end{proof}

Then it follows the characterization of simplicity.

\begin{corol}\label{theor_simple_corol} 
Let $A$ be a $C^*$-algebra and let $\beta:A\rightarrow A$ be an extendible injective endomorphism with $\beta(A)$ being a hereditary sub-$C^*$-algebra of $A$. Then $A\times_\beta \N$ is simple if and only if $A$ is $\beta$-simple and $\beta^n$ is outer for every $n>0$.
\end{corol}
\begin{proof}
$A\times_\beta \N$ is simple if and only if $A_\infty\times_{\beta_\infty} \Z$ is simple  if and only if $A_\infty$ is $\beta_\infty$-simple and $\mathbb{T}(\beta_\infty)=\mathbb{T}$  \cite[Theorem 6.5]{OP1} if and only if $A$ is $\beta$-simple and $\mathbb{T}(\beta)=\mathbb{T}$. Therefore, by Theorem \ref{theor_simple} we have that $A$ is $\beta$-simple and $\mathbb{T}(\beta)=\mathbb{T}$ if and only if $A$ is $\beta$-simple and $\beta^n$ is outer for every $n>0$.
\end{proof}

\begin{exem} 
The following example is \cite[Theorem 9.1]{OP3}. Let $A$ be a $C^*$-algebra with a faithful bounded trace, let $A\longrightarrow \mathcal{B}(\mathcal{H})$ be a faithful non-degenerate representation of $A$, and let $V$ be a non-unitary isometry of $\mathcal{B}(\mathcal{H})$ with $VV^*\in M(A)$ such that $VAV^*+V^*AV\subseteq A$. Then suppose that there are no non-trivial ideals $I$ of $A$ such that $VIV^*+V^*IV\subseteq I$. Then we claim that the $C^*$-algebra  $B:=C^*(\{AV^n(V^*)^m\}_{n,m\geq 0})\subseteq\mathcal{B}(\mathcal{H})$ is simple. Indeed, let us define the endomorphism $\beta:A\longrightarrow A$ by $\beta(a)=VaV^*$ for every $a\in A$ that is extendible (since $VV^*\in M(A)$). Clearly satisfies that $\beta(A)$ is a hereditary sub-$C^*$-algebra of $A$, and it does not have any non-trivial $\beta$-invariant ideal. Now, since $\tau$ is a faithful bounded trace of $A$, we can extended it to a faithful bounded trace $\bar{\tau}$ of $M(A)$. Hence, $M(A)$ has no non-unitaries isometries. Therefore, by Theorem \ref{theor_simple} $A\times_\beta\N$ is simple, whence the natural map $A\times_\beta\N\longrightarrow B$ is an isomorphism.
\end{exem}

\section{Graph $C^*$-algebras}

In this section, we  apply the above results to determine the simplicity of certain graph $C^*$-algebras. Though their simplicity is well understood in terms of properties of the graph, we are going to deduce it from the properties of their associated $C^*$-dynamical systems. 

We use the conventions of \cite{Rae}. Let $E=(E^0,E^1,r,s)$ be a countable directed graph; $r,s:E^1\rightarrow E^0$ denote the \emph{range} and the \emph{source} maps of an edge. We say that $E$ is \emph{column-finite} if $|s^{-1}(v)|<\infty$ for every $v\in E^0$.   A vertex $v\in E^0$ is a \emph{sink} (\emph{source}) if $|s^{-1}(v)|=0$ ($|r^{-1}(v)|=0$). A vertex $v\in E^0$ is called \emph{singular} if is either a source or an infinite receiver. We denote by $E^0_{sing}$ the set of all singular vertices.  A \emph{path} $\alpha$ of length $n$ is a concatenation of $n$ edges $e_n\cdots e_1$ with $r(e_i)=s(e_{i+1})$ for $i=1,...,n-1$. Given a path $\alpha$ we denote by $|\alpha|$ its length. Let $E^n$ be the set of all paths of length $n$, and $E^*=\cup_{n\geq 0} E^n$ the set of all the paths of finite length in $E$. Finally, given $\alpha,\eta\in E^*$, we say that $\alpha\in \eta$ if there exist $\rho,\gamma\in E^*$ such that $\eta=\rho\alpha\gamma$.

Recall that the \emph{graph $C^*$-algebra} $C^*(E)$ is the universal $C^*$-algebra generated by  orthogonal projections $\{P_v\}_{v\in E^0}$ and partial isometries $\{S_e\}_{e\in E^1}$, satisfying the following conditions:
\begin{align*}
(CK1) \qquad S^*_eS_f=\delta_{e,f}\cdot P_{s(e)} \qquad \text{for every }e,f\in E^1 \\
(CK2) \qquad P_v=\sum_{r(e)=v}S_eS^*_e\qquad \text{ for every }v\in  E^0_{sing}.
\end{align*}
See \cite{Rae} for a survey on graph $C^*$-algebras. One can naturally define a group homomorphism $\gamma: \mathbb{T}\rightarrow \text{Aut } C^*(E)$, given by $\gamma_z(P_v)=P_v$ and $\gamma_z(S_e)=zS_e$ for every $z\in  \mathbb{T}$, $v\in E^0$ and $e\in E^1$; it is the so-called \emph{gauge action} on $C^*(E)$. An ideal $I$ of $C^*(E)$ is said to be a \emph{gauge invariant ideal} if $\gamma_z(I)=I$ for every $z\in \mathbb{T}$ (see \cite{Bates} and \cite{Bates2}). The \emph{core} sub-$C^*$-algebra of $C^*(E)$ is defined as
$$C^*(E)^\gamma:=\{x\in C^*(E): \gamma_z(x)=x\text{ for every }z\in \mathbb{T}\}\,.$$
We can give another description of the core. For every $n\in\N$ and $v\in E^0$, define
$$\mathcal{F}_n(v):=\{S_\eta S_\rho^*: \eta,\rho\in E^n\text{ with }s(\eta)=s(\rho)=v\}\cong M_{k_{n,v}}(\C)$$
for some $k_{n,v}\in \N$, and let $\mathcal{F}_n=\oplus_{v\in E^0}\mathcal{F}_n(v)$. Now, if we index the vertices $\{v_i\}_{i\geq 0}$, then we define $C_{n.m}:=\sum_{i,j\geq 0}^{n,m}\mathcal{F}_i(v_j)$ for every $n,n\geq 0$. These are finite dimensional sub-$C^*$-algebras of $C^*(E)^\gamma$ with $C_{n,m}\subseteq C_{n,m+1}$ and $C_{n,m}\subseteq C_{n+1,m}$ for every $n,m\geq 0$. Hence,
$$C^*(E)^\gamma=\overline{\bigcup_{n,m\geq 0} C_{n,m}}$$
is an AF-algebra. 

We recall the following result from \cite{AR}, that allows to present certain graph $C^*$-algebras as $C^*$-dynamical systems (we would like to thank the authors for showing us the result even before the releasing of the manuscript). 

\begin{theor}[{\cite[Theorem  9.3]{AR}}]
Let $E$ be a  column finite graph without sinks. If we define the endomorphism $\beta_E:C^*(E)^\gamma\rightarrow C^*(E)^\gamma$ as $\beta_E(z)=TzT^*$ for every $z\in C^*(E)^\gamma$, where 
 $$T=\sum_{e\in E^1}|s^{-1}(s(e))|^{1/2} S_e$$ is an isometry of $M(C^*(E))$, then we have that $C^*(E)\cong C^*(E)^\gamma \times_{\beta_E} \N$.
\end{theor}

Notice that the endomorphism $\beta_E:C^*(E)^\gamma\rightarrow C^*(E)^\gamma$ is injective, extendible and $\beta_E(C^*(E)^\gamma)$ is a hereditary sub-$C^*$-algebra of $C^*(E)^\gamma$.

\begin{defi} 
A subset $H\subseteq E^0$ is said to be \emph{hereditary} if, whenever $\eta\in E^*$ with $r(\eta)\in H$, then $s(\eta)\in H$. We say that $H$ is \emph{saturated} if, whenever $|r^{-1}(v)|<\infty$ and $\{s(r^{-1}(v))\}:=\{z\in E^0: z=s(e)\text{ for some }e\in r^{-1}(v)\}\subseteq H$, then $v\in H$. 
\end{defi}

By \cite[Theorem 4.1]{Bates}, there exists a bijection between hereditary and saturated subsets of $E^0$ and gauge invariant ideals of $C^*(E)$, $H\longmapsto K_H$, where $K_H:=\overline{\text{span }}\{S_{\eta}S^*_{\nu}: \eta,\nu\in E^*\text{ with }s(\eta)=s(\nu)\in H\}$. The inverse map is
$K\longmapsto H_K$, where $H_K:=\{ v\in E^0 : P_v\in K\}$.

Now, given a hereditary and saturated subset of $E^0$, we define
$$I_H:=K_H\cap {C^*(E)}^{\gamma}.$$
By Remark \ref{desprescaldra}, $I_H$ is a $\beta_E$-invariant ideal of ${C^*(E)}^{\gamma}$, and it is easy to see that
$$I_H:=\overline{\sum_{v\in H, n\geq 0}\mathcal{F}_n(v)}.$$
On the other side, if $K$ is a gauge invariant ideal of $C^*(E)$, since $I:=K\cap {C^*(E)}^{\gamma}$ is a $\beta_E$-invariant ideal, then we have that the set
$$H_I:=\{v\in E^0: P_v\in I\}$$ 
is a subset of $H_K$. Moreover, since $P_v\in {C^*(E)}^{\gamma}$ for every $v\in E^0$, it is clear that $H_K\subseteq H_I$, whence $H_K=H_I$. Thus, $H_I$ is an hereditary and saturated subset of $E^0$.

In particular, if $I$ is a $\beta_E$-invariant ideal of ${C^*(E)}^{\gamma}$, then $K:=\overline{(C^*(E))\cdot I \cdot (C^*(E))}$ is a gauge invariant ideal of $C^*(E)$ and
$$I_{H_I}=I_{H_K}=K_{H_K}\cap {C^*(E)}^{\gamma}=K\cap {C^*(E)}^{\gamma}=I.$$
Conversely, if $H$ is a hereditary and saturated subset of $E^0$, since $I_H:=K_H\cap {C^*(E)}^{\gamma}$, we conclude that
$$H_{I_H}=H_{K_H}=H.$$

Summarizing, there exists a bijection between the hereditary and saturated subsets of $E^0$ and the $\beta_E$-invariant ideals of ${C^*(E)}^{\gamma}$ defined by the maps
$$H\longmapsto I_H=\overline{\sum_{v\in H, n\geq 0}\mathcal{F}_n(v)} \qquad \text{and} \qquad I\longmapsto H_I=\{v\in E^0: P_v\in I\}. $$

One could be tempted to think that there is a bijection between hereditary sets of $E^0$ and the ideal of ${C^*(E)}^{\gamma}$ such that $\beta_E(I)\subseteq I$, but this is not the case (see Examples \ref{exem1}). 

\begin{theor}[{cf. \cite[Theorem 4.1]{Bates}}] 
If $E$ is a column  finite graph without sinks, then there is a bijection between the closed gauge invariant ideals of $C^*(E)$, the hereditary and saturated subsets of $E^0$ and the $\beta_E$-invariant ideals of ${C^*(E)}^{\gamma}$.
\end{theor}

\begin{corol}\label{gauge_inv_simple} Let $E$ be  a column finite graph without sinks, then $E^0$ has no non-trivial hereditary and saturated subsets if and only if $C^*(E)^\gamma$ does not have a proper ${\beta_E}$-invariant  ideals.
\end{corol}

\begin{exem}\label{exem1} In the following examples we would like to illustrate some consequences of Corollaries \ref{theor_simple_corol} and \ref{gauge_inv_simple} and determine the simplicity of some graph $C^*$-algebras. We would like to remark again that this is well-known  by \cite[Proposition 5.1]{Bates}. However, one can slightly modify some of the examples to get new simple $C^*$-algebra that probably do not arise as graph $C^*$-algebras.

\begin{enumerate}
\item Consider the graph $E$
$$\xymatrix{{\bullet_{v_0}} \ar[r] & {\bullet_{v_1}}\ar[r]& {\bullet_{v_2}}\ar[r] & {\cdots} }$$
Then $E^0$ has no non-trivial hereditary and saturated subsets. We have that ${C^*(E)}^{\gamma}\cong C_0(\N\cup\{0\})$ and the endomorphism ${\beta_E}:{C^*(E)}^{\gamma}\rightarrow {C^*(E)}^{\gamma}$  sends $\chi_{\{i\}}$ (the characteristic function at $i$) to $\chi_{\{i+1\}}$ for every $i\geq 0$.  Then, since $E^0$ does not have non-trivial saturated and hereditary subsets, $C_0(\N\cup\{0\})$ is $\beta_E$-simple. Moreover, since $M(C_0(\N\cup\{0\}))$ is a commutative $C^*$-algebra, it does not have non-unitary isometries. Hence, $C^*(E)$ is simple.

\item Consider the graph $E$
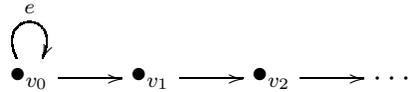
$$\xymatrix{{\bullet_{v_0}}\uloopr{e} \ar[r] & {\bullet_{v_1}}\ar[r]& {\bullet_{v_2}}\ar[r] & {\cdots} }$$
Then ${C^*(E)}^{\gamma}=\mathcal{K}$ (the compact operators of a countable infinite dimensional Hilbert space $\mathcal{H}$), that is simple. Therefore $C^*(E)^\gamma$ is ($\beta_E-$) simple, and thus $E^0$ has no non-trivial hereditary and saturated subsets. Moreover, it is not difficult to see that $\beta_E=\text{Ad }W$ where $W$ is the shift operator of $\mathcal{H}$, whence $\beta_E$ is inner and $C^*(E)$ is not a simple $C^*$-algebra. 

\item This is the graph $C^*$-algebra picture of the algebra $\mathcal{O}_n$. Let $E$ be the graph
$$\xymatrix{{\bullet_{v}}\uloopr{(n)} }$$
with $n$ loops. We have that ${C^*(E)}^{\gamma}$ is isomorphic to the $n$-infinity UHF-algebra $\mathcal{U}_n:=\bigotimes_{i=1}^\infty M_n$ and $\beta_E(x)=P\otimes x$ for every $x\in \mathcal{U}_n$, where 
$P=  \left(\begin{array}{ccc} 1/n & \cdots & 1/n  \\ \vdots & &\vdots \\ 
1/n & \cdots &1/n \end{array}\right)$. Therefore ${C^*(E)}^{\gamma}$ is $\beta_E$-simple, since $\mathcal{U}_n$ is simple. Moreover, since $C^*(E)^\gamma$ is a unital and finite $C^*$-algebra  it does not have non-unitary isometries, and therefore $C^*(E)$ is a simple $C^*$-algebra.

\item An example of the different behaviour of $\beta_E$ and $\beta^2_E$ can be found when the graph  $E$ is
\[
{
\def\labelstyle{\displaystyle}
 \xymatrix{ {\bullet_v}\ar@/^4pt/ [r] \ar@/^8pt/ [r]   & \bullet_w \ar@/^4pt/ [l] }
}\]
In this case $P_v$ and $P_w$ generate two orthogonal ideals of ${C^*(E)}^{\gamma}$, $I_v$ and $I_w$ respectively, both isomorphic to the CAR-algebra $\bigotimes_{n=1}^{\infty} M_2$, and such that ${C^*(E)}^{\gamma}=I_v\oplus I_w$. We have that $\beta_E(x,y)=(y,P\otimes x)$ for every $(x,y)\in I_v\oplus I_w$, where 
$P=  \left(\begin{array}{cc} 1/2 &  1/2  \\ 
1/2 & 1/2 \end{array}\right)$. Therefore ${C^*(E)}^{\gamma}$ is $\beta_E$-simple, but $\beta^2_E(I_v)\subseteq I_v$ and $\beta^2_E(I_w)\subseteq I_w$. Moreover, since $C^*(E)^\gamma$ is a unital and finite $C^*$-algebra,  it does not have non-unitary isometries, and therefore $C^*(E)$ is a simple $C^*$-algebra.

\item Consider the graph $E$ 
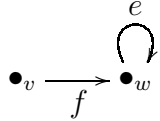
\[
{
\def\labelstyle{\displaystyle}
 \xymatrix{\bullet_v \ar[r]_{f}  & \bullet_w \uloopr{e}}
}\]
Observe that $C^*(E)^\gamma=\overline{\text{span }} \{P_v,\,S_{e^m}S^*_{e^m},\, S_{e^mf}S^*_{e^mf}: m\geq 0 \}$. Then, $C^*(E)^\gamma$ is a commutative $C^*$-algebra isomorphic to $C(X)$ where $X=\{1/n: 1\leq n\}\cup\{0\}$, and the endomorphism acts $\beta_E(\chi_{[0,1/n]})= \chi_{[0,1/(n+1)]}$ for every $n\geq 1$.
Since  $\{v\}$ is a saturated and hereditary subset, there exists a proper  ${\beta_E}$-invariant ideal, that corresponds to the ideal $C_0(X\setminus\{0\})$. Given $n\in\N$ let $I_n$ be the ideal of ${C^*(E)}^{\gamma}$ generated by $\chi_{[0,1/n]}$. Observe that  $\beta_E(I_n)\subseteq I_n$ (in particular $\beta_E(I_n)=I_{n+1}$). Therefore ${C^*(E)}^{\gamma}$ posses a infinite countably family of different ideals $I$ such that $\beta_E(I)\subseteq I$.  
\end{enumerate}
\end{exem}

\begin{defi} 
A $C^*$-dynamical system $(A,\beta)$ is said to satisfy the \emph{Cuntz-Krieger uniqueness theorem} if for every faithful Stacey covariant representation $(\pi,V)$ of  $(A,\beta)$  we have that $C^*(\pi,V)\cong A\times_\beta \N$.
\end{defi}

Recall that the graph $E$ satisfies \emph{condition (L)} if every cycle has an entry. A  graph $E$ satisfies condition (L) if and only if, given any $*$-homomorphism $\eta:C^*(E)\rightarrow B$ such that $\eta(P_v)\neq 0$ for every $v\in E^0$, we have that $\eta$ is injective (see e.g. \cite[Section 2]{Rae}). Thus, 

\begin{theor}[{cf. \cite[Theorem 2.4]{Rae} \& \cite[Theorem 2.5]{OP2}}]\label{graph_CK}	 
Let $E$ be a column finite graph  without sinks. Then the following statements are equivalent:
\begin{enumerate}
\item The graph $E$ satisfies condition $(L)$.
\item $({C^*(E)}^{\gamma}, \beta_E)$ satisfies the Cuntz-Krieger uniqueness theorem.
\item $\mathbb{T}(\beta_E)=\mathbb{T}$.
\end{enumerate}
\end{theor}

Now we will see that for the dynamical system $({C^*(E)}^{\gamma}, \beta_E)$ associated to a graph $C^*$-algebras, the results of Olesen and Pedersen \cite[Theorem 2.5 \& Theorem 4.6]{OP2} reduces to a simpler way.

\begin{prop}\label{prop2} 
Let $E$ be a column finite graph without sinks and let $({C^*(E)}^{\gamma}, \beta_E)$ be its associated $C^*$-dynamical system. If $\beta_E$ is weakly properly outer then $E$ satisfies condition (L).
\end{prop}
\begin{proof}
Suppose that  $\beta_E$ is weakly properly outer and that $E$ does not satisfy condition (L), i.e., there exists a cycle $\alpha$ without an entry. We can suppose that $\alpha=e_n\cdots e_1$ with $e_i\in E^1$ and $v_i=r(e_i)$ for $i=1,\ldots,n$, such that $s(e_1)=r(e_n)=v_n$ and $r(e_i)\neq v_n$ for every $i\neq n$. 
Let $H_\alpha=\{v_i\}_{i=1,\ldots,n}$, and let $I$ be the ideal of ${C^*(E)}^{\gamma}$ generated by $\{P_v\}_{v\in H_\alpha}$. Observe that, since $\alpha$ does not have any exit, by (CK2) we have that 
$$I=\overline{\sum_{k\geq 0}\mathcal{F}_k(v_n)  }\,.$$ 

Given $w\in E^0$, let $\{\eta_{i,w}\}_{i=1}^{\nu_w}\subseteq E^n$ be the set of paths such that $s(\eta_{i,w})=w$ (a finite number since $E$ is column finite),. Given any $z\in E^0_{v_n}$, where 
$$E^0_{v_n}:=\{z\in E^0: \text{ exists }\eta\in E^* \text{ with }s(\eta)=v_n \text{ and } r(\eta)=z\}\,,$$
consider all the paths $\{\gamma_{j,z}\}_{j\in\Delta_z}\subseteq E^*$ such that $s(\gamma_{j,z})=v_n$ and $r(\gamma_{j,z})=z$. Observe that $1\leq |\{\gamma_{j,z}\}_{j\in\Delta_z}|\leq \infty$. Given any path $\gamma_{j,z}$, we define $\kappa_{i,z}:=|s^{-1}(s(f_n))|\cdots |s^{-1}(s(f_1))|<\infty$ for $f_n\cdots f_1=\eta_{i,z}$ with $f_i\in E^1$. Then, define the formal sums (we still not determine where their converge to)
 $$V_w:=\sum_{j\in \Delta_w, i=1}^{\nu_w} \kappa_{i,w}^{-1/2}S_{\eta_{i,w}\gamma_{j,w}}S^*_{\gamma_{j,w}\alpha}\qquad \mbox{ if }w\in E^0_{v_n}\setminus H_\alpha\,,$$
$$V_{v_k}=\sum_{i=1 }^{\nu_{v_k}} \kappa_{i,v_k}^{-1/2} S_{\eta_{i,v_k} e_k\cdots e_1}   S^*_{e_k\cdots e_1\alpha} \qquad \mbox{ for }1\leq k\leq n-1$$
and 
$$V_{v_n}=\sum_{i=1 }^{\nu_{v_n}} \kappa_{i,v_n}^{-1/2} S_{\eta_{i,v_n}}   S^*_{\alpha}\,.$$
We claim that $\sum_{w\in E^0_{v_n}} V_w$ converges with the strong topology in $M(I)$. Indeed, recall that $I=\overline{\text{span }}\{S_{\gamma_{i,w}} S^*_{\gamma_{j,z}}: |\gamma_{i,w}|=|\gamma_{j,z}|\text{ for }z,w\in E^0_{v_n}\}$, so it is enough to see that for every $v,w\in E^0_{v_n}$ and $k,l\in\N$ such that  $|\gamma_{k,w}|=|\gamma_{l,z}|$ then $(\sum_{u\in E^0_{v_n}} V_u)S_{\gamma_{k,w}} S^*_{\gamma_{l,z}}$ and $S_{\gamma_{k,w}} S^*_{\gamma_{l,z}}(\sum_{u\in E^0_{v_n}} V_u)$ are elements of $I$ of norm less or equal to $1$.  Observe that 
$$(\sum_{u\in E^0_{v_n}} V_u)S_{\gamma_{k,w}} S^*_{\gamma_{l,z}}=V_{w}S_{\gamma_{k,w}} S^*_{\gamma_{l,z}}=\sum_{i=1}^{\nu_w} \kappa_{i,w}^{-1/2}S_{\eta_{i,w}\gamma_{k,w}}S^*_{\gamma_{l,z}\alpha}\in I\,.$$

Now we have that 
\begin{align*}
\|\sum_{i=1}^{\nu_w} \kappa_{i,w}^{-1/2}S_{\eta_{i,w}\gamma_{k,w}}S^*_{\gamma_{l,z}\alpha}\|^2 & =\|(\sum_{i=1}^{\nu_w} \kappa_{i,w}^{-1/2}S_{\eta_{i,w}\gamma_{k,w}}S^*_{\gamma_{l,z}\alpha})^*\sum_{i=1}^{\nu_w} \kappa_{i,w}^{-1/2}S_{\eta_{i,w}\gamma_{k,w}}S^*_{\gamma_{l,z}\alpha}\| \\
& =\|\sum_{i=1}^{\nu_w} \kappa_{i,w}^{-1}S_{\gamma_{l,z}\alpha}S^*_{\gamma_{l,z}\alpha}\|\leq \|S_{\gamma_{l,z}}S^*_{\gamma_{l,z}}\|=1
\end{align*}
as desired. Analogously, we can see that $S_{\gamma_{k,w}} S^*_{\gamma_{l,z}}(\sum_{u\in E^0_{v_n}} V_u)$ converges to an element of $I$ of norm less or equal than $1$. Define $V:=\sum_{w\in E^0_{v_n}} V_w$. Then we have that $V^*V=1_{M(I)}$. One can easily check that $VzV^*={\beta_E}^n(z)$ for every $z\in I$, so ${\beta_E}^n_{|I}$ is inner, a contradiction with the hypothesis.
\end{proof} 

\begin{rema}\label{minimal} 
In the proof of Proposition \ref{prop2} we prove that, if $E$ does not satisfy condition (L), i.e., there exists a loop $\alpha$ without entries, then given any vertex $v\in E^0$ of the loop $\alpha$ we have that  $$I:=\overline{\sum_{n\geq 0} \mathcal{F}_k(v)} \,,$$  is a  $\beta_E$-invariant ideal of ${C^*(E)}^{\gamma}$   such that ${\beta_E}^n_{|I}$ is inner for some $n>0$. Then $H_{I}$ is a hereditary and saturated subset of $E^0$ containing $H_\alpha:=\{v\in E^0: v\in \alpha\}$ such that $I=I_{H_{I}}$.  But since $I_{H_I}\subseteq I_H$ for every hereditary and saturated subset $H$ of $E^0$  containing $H_{\alpha}$ we have that $H_I$ is the minimal hereditary and saturated subset of $E^0$ containing $H_\alpha$. Thus $I$ is a minimal $\beta_E$-invariant ideal such that  ${\beta_E}^n_{|I}$ is inner.

Observe also that in general $I$ is a non-simple sub-$C^*$-algebra of ${C^*(E)}^{\gamma}$. For example, if $n>1$, then the ideal generated by $P_{v_i}$ for every $1\leq i\leq n$ is proper (one can check easily that it cannot contain $P_{v_j}$ for $i\neq j$). 
\end{rema}

\begin{prop}\label{prop3} 
Let $E$ be a column finite graph without sinks and let $({C^*(E)}^{\gamma}, \beta_E)$ be its associated $C^*$-dynamical system. If $E$ satisfies condition (L) then $\beta_E$ is weakly properly outer.
\end{prop}
\begin{proof}
Suppose that $E$ satisfies condition (L) and there exist a non-zero ${\beta_E}$-invariant ideal $I$ of ${C^*(E)}^{\gamma}$ such that ${\beta^n_E}_{|I}$ is inner for some $n>0$. Hence, $I=\overline{\sum_{v\in H_I, k\geq 0} \mathcal{F}_k(v)}$, where $H_I$ is a hereditary and saturated subset of $E^0$. So, there is a gauge-invariant ideal of $C^*(E)$, say $K_{H_I}$, generated by $\{P_v\}_{v\in H_I}$. Recall that the core $K_{H_I}^\gamma$ is precisely $I$. Now, there exists an isometry $W\in M(I)$ such that ${\beta_E}^n(z)=WzW^*$ for every $z\in I$. Since $I$ contains an approximate unit for $K_{H_I}$ (see for example \cite[Lemma 3.4]{AR2}) we can see $M(I)$ as a sub-$C^*$-algebra of $M(K_{H_I})$.

Define $U:=W^*T^n$ a unitary in $M(K_{H_I})$ (we are also using that $TI\subseteq IT$).  Then, for every $z\in I$ we have $z=UzU^*$
Observe that, given any $y\in I$, we have that that $yU=Uy$.

We claim that $H_I\subseteq E^0$ cannot have sources. Indeed, if $v\in H_I$ is a source, then $P_v\in I$, and hence 
$$P_v=UP_vU^*=W^*T^n P_vT^{n*} W\,.$$
We have that $W^*T^n P_v=\sum_{l}\lambda_l S_{\eta_l}S^*_{\rho_l}\in P_v K_{H_I}P_v\subseteq P_v C^*(E)P_v$, with $\lambda_l\in \C$  and $\eta_l,\rho_l\in E^*$ with $r(\eta_l)=r(\rho_l)=v$ and $|\eta_l|=|\rho_l|+n$. But this contradicts the fact that $v$ is a source. 

Now, given any $v\in E^0$, let  $\{\eta_{i,v}\}_{i=1}^{\nu_v}$ be the set of all the paths in $E^n$ with $r(\eta_{i,v})=v$. So, given $v\in H_I$ and $i\leq \nu_v$, we define
$$0\neq X_{i,v}:=S^*_{\eta_{i,v}}U\in P_{s(\eta_{i,v})}{C^*(E)}^{\gamma}P_v\,.$$ 
If $\mu, \gamma\in E^m$  with $s(\mu)=s(\gamma)\in H_I$ then we have that $S_{\mu}S^*_\gamma\in I$. So, for every $w,z\in H_I$, $1\leq k \leq \nu_w$ and $1\leq l\leq \nu_z$, we have that
\begin{enumerate}
\item $X_{k,w}(S_{\mu}S^*_\gamma) X^*_{l,z}=S^*_{\eta_{k,w}}U(S_{\mu}S^*_\gamma)U^* S^*_{\eta_{l,z}}= 
S^*_{\eta_{k,w}}(S_{\mu}S^*_\gamma)S_{\eta_{l,z}}$, 
\item $X_{k,w} X^*_{l,z}=\delta_{w,z} \cdot \delta_{k,l} \cdot P_{s(\eta_{k,w})}$,
\item $X^*_{k,w} S_{\mu}S^*_{\gamma}X_{l,z}  =U^*S_{\eta_{k,w}}S_{\mu}S^*_{\gamma}S^*_{\eta_{l,z}} U =\delta_{r(\mu),s(\eta_{k,w})} \cdot \delta_{r(\gamma),s(\eta_{j,z})} \cdot  S_{\eta_{k,w}}S_{\mu}S^*_{\gamma}S^*_{\eta_{l,z}}$,
\end{enumerate}
while given $v\in H_I$ and $1\leq i\leq\nu_v$, we have that 
$$X_{i,v}X^*_{i,v}X_{i,v}=S^*_{\eta_{i,v}}UU^*S_{\eta_{i,v}}S^*_{\eta_{i,v}}U =S^*_{\eta_{i,v}}S_{\eta_{i,v}}S^*_{\eta_{i,v}}U=S^*_{\eta_{i,v}}U=X_{i,v}\,,$$
so $X_{i,v}$ is a partial isometry in $P_{s(\eta_{i,v})}{C^*(E)}^{\gamma}S_{\eta_{i,v}}S^*_{\eta_{i,v}}$.

Now, choose any $v\in H_I$ and $1\leq i\leq  \nu_v$, and  consider the isometry $X_{i,v}S_{\eta_{i,v}}\in P_{s(\eta_{i,v})} C^*(E)P_{s(\eta_{i,v})}$. Given any $\varepsilon>0$, there exist $m\in\N$, $\lambda_j\in \C\setminus\{0\}$ and $\alpha_j,\beta_j\in E^*$ with $|\alpha_j|=|\beta_j|+n$ and $r(\alpha_j)=r(\beta_j)=s(\eta_{i,v})$  such that $$\| X_{i,v}S_{\eta_{i,v}} - \sum_{j=1}^m \lambda_j S_{\alpha_j} S^*_{\beta_j}\|<\varepsilon\,.$$
Suppose that $|\beta_1|\geq |\beta_i|$ for every $i\leq m$. Then we have that 
\begin{align*} 
\|p_{s(\eta_{i,v})} - y^*y\| & =\| (X_{i,v}S_{\eta_{i,v}})^*X_{i,v}S_{\eta_{i,v}} - y^*y\| \\
& \leq \| (X_{i,v}S_{\eta_{i,v}})^*X_{i,v}S_{\eta_{i,v}} - y^*X_{i,v}S_{\eta_{i,v}}\|+\| y^*X_{i,v}S_{\eta_{i,v}} - y^*y\| \\
& \leq \| (X_{i,v}S_{\eta_{i,v}})^* - y^*\|\|X_{i,v}S_{\eta_{i,v}}\|+\| X_{i,v}S_{\eta_{i,v}} - y\|\|y^*\| \\
& \leq \varepsilon\cdot 1 + \varepsilon (1+\varepsilon)
\end{align*}
Thus, if $\varepsilon<1/4$, we have that $y^*y$ is invertible in $P_{s(\eta_{i,v})}{C^*(E)}^{\gamma}P_{s(\eta_{i,v})}$. Hence, $yS_{\beta_j}\neq 0$ for every $1\leq j\leq m$. Thus, 
\begin{align*}
\| S_{\beta_1} S^*_{\beta_1} - yS_{\beta_1} S^*_{\beta_1}y^*\| & =\|(X_{i,v}S_{\eta_{i,v}}) S_{\beta_1} S^*_{\beta_1}(X_{i,v}S_{\eta_{i,v}})^*-yS_{\beta_1} S^*_{\beta_1}y^*\|  \\
 & \leq   \|(X_{i,v}S_{\eta_{i,v}}) S_{\beta_1} S^*_{\beta_1}(X_{i,v}S_{\eta_{i,v}})^*-(X_{i,v}S_{\eta_{i,v}}) S_{\beta_1} S^*_{\beta_1}y^*\|+  \\
  & \qquad + \|(X_{i,v}S_{\eta_{i,v}}) S_{\beta_1} S^*_{\beta_1}y^*-yS_{\beta_1} S^*_{\beta_1}y^*\| \\
   &  \leq \|(X_{i,v}S_{\eta_{i,v}}) S_{\beta_1} S^*_{\beta_1}\| \varepsilon + \|S_{\beta_1} S^*_{\beta_1}y^*\|\varepsilon \\
   & \leq \varepsilon + \varepsilon (1+\varepsilon)
\end{align*}
and therefore $yS_{\beta_1} S^*_{\beta_1}y^*$  is invertible in $S_{\beta_1} S^*_{\beta_1} {C^*(E)}^{\gamma}S_{\beta_1} S^*_{\beta_1} $. So, $0\neq S^*_{\beta_1}yS_{\beta_1} =\sum_{j=1}^{m'} \lambda'_j S_{\gamma_j}$, where $\gamma_j\in E^n$ with $s(\gamma_j)=r(\gamma_j)=r(\beta_1)=s(\eta_{i,v})$ for every $1\leq j\leq m'$. Hence, the $\gamma_j$s  are cycles. Let $\gamma=\gamma_1$. Since by assumption $E$ satisfies condition (L), we have  that $\gamma$ has an entry. Therefore, there exists $\eta\in E^*$ such that $\gamma \notin \eta$. So, we have that $S_{\beta_1 \eta} S^*_{\beta_1 \eta}\in  S_{\beta_1} S^*_{\beta_1} {C^*(E)}^{\gamma}S_{\beta_1} S^*_{\beta_1} $ and hence $S^*_{\beta_1 \eta} S_{\beta_1}=S^*_{\eta} S_{\gamma}=0$, that contradicts the fact that $S_{\beta_1}S^*_{\beta_1}$ is invertible in $S_{\beta_1} S^*_{\beta_1} {C^*(E)}^{\gamma}S_{\beta_1} S^*_{\beta_1} $.
\end{proof}

Summarizing, we have the following result.

\begin{theor}\label{graph_CK_outer}	 
Let $E$ be a column finite graph  without sinks. Then the following statements are equivalent:
\begin{enumerate}
\item The graph $E$ satisfies condition $(L)$.
\item $({C^*(E)}^{\gamma}, \beta_E)$ satisfies the Cuntz-Krieger uniqueness theorem.
\item $\mathbb{T}(\beta_E)=\mathbb{T}$.
\item There is no $\beta_E$-invariant ideal $I$ of $C^*(E)^\gamma$ and $n\in \N$ such that ${\beta^n_E}_{|I}=\text{Ad }V$, where $V\in M(I)$ is an isometry.
\item $\beta_E$ is weakly properly outer.
\end{enumerate}
\end{theor}
 
Finally, using the characterization of simplicity of $C^*(E)$ in terms of properties of the graph $E$ \cite[Proposition 5.1]{Bates}, the representation of an Huef and Raeburn of the graph $C^*$-algebra $C^*(E)$ as the Stacey crossed product ${C^*(E)}^{\gamma}\times_{\beta_E} \N$, joint with Corollary \ref{theor_simple_corol} and Theorem \ref{graph_CK_outer}, we conclude the desired result.

\begin{theor}[{cf. \cite[Proposition 5.1]{Bates}}]\label{theor1} 
Let $E$ be a column finite graph without sinks. Then the following statements are equivalent:
\begin{enumerate}
\item ${C^*(E)}^{\gamma}\times_{\beta_E} \N$ is simple.
\item $E$ does not have non-trivial hereditary and saturated subsets and satisfies condition (L).
\item ${C^*(E)}^{\gamma}$ does not have any proper ${\beta_E}$-invariant ideal and $\beta^n_E$ is outer for every $n\geq 1$. 
\end{enumerate}
\end{theor}

\section{Pure infiniteness}

In Theorem \ref{theor_simple} we have given necessary and sufficient conditions on the endomorphism $\beta$ for the simplicity of the $C^*$-algebra $A\times_\beta\N$. 
If $A$ is a unital $C^*$-algebra and $\beta(1)\neq 1$ we have then that $A\times_\beta\N$ contains a proper isometry, and if in addition  $A\times_\beta\N$ is simple, we have that it is a properly infinite $C^*$-algebra.
We will see that for a broad class of unital real rank zero $C^*$-algebras $A$ we have that  $A\times_\beta\N$ turns out to be purely infinite. Our result generalize and unify similar results given in \cite{Ror} and \cite{JKO}. 

\begin{lem}\label{full_element}
Let $A$ be a unital $C^*$-algebra, let $\beta:A\longrightarrow A$ be an injective endomorphism, and suppose that does not exist  any proper ideal $I$ of $A$ such that $\beta(I)\subseteq I$. Then, given any non-zero $a\in A_+$ there exists $n\in \N$ such that $a+\beta(a)+\cdots+\beta^n(a)$ is a full positive element in $A$.
\end{lem}
\begin{proof} Consider the ideal $I:=\overline{\text{span }}\{x\beta^n(a)y: n\geq 0,\, x,y\in A\}\neq 0$. It clearly satisfies $\beta(I)\subseteq I$ and then, by hypothesis, we have that $I=A$. Therefore  we can write 
$$1=\sum_{i=1}^kx_i\beta^{n_k}(a)y_k$$
where $x_i,y_i\in A$ and $n_i\in \N$ for every $i\in\{1,\ldots,k\}$. Then, taking $n=\text{max}_i \{n_i\}$, we have or desired result.
\end{proof}

Let $T(A)$ be the set of tracial states of $A$, which is a compact space with the $*$-weak topology. We say that $A$ has \emph{strict comparison} if: (i) $T(A)\neq \emptyset$; (ii) Whenever $p\in \overline{AqA}$ such that $\tau(p)<\tau(q)$ for every $\tau\in T(A)$, we have that $p\lesssim q$.  For example, every unital exact and stably finite $C^*$-algebra of real rank zero that is  $\mathcal{Z}$-stable has strict comparison \cite[Corollary 4.10]{Ror2}.

Recall that a (non-necessarily simple) $C^*$-algebra $A$ is said to be \emph{purely infinite} if and only if all positive elements are properly infinite \cite{KR} ; in particular, every projection of $A$  (if it has any) must be properly infinite. Also recall that a unital simple $C^*$-algebra is purely infinite if and only if has real rank zero and every projection is infinite \cite{Z2}. The following lemma is a slight modification of \cite[Lemma 3.2]{Ror}.

\begin{lem}[{cf. \cite[Lemma 3.2]{Ror}}]\label{comparison} Let $A$ be a unital $C^*$-algebra that either has strict comparison or is purely infinite. Let $\beta:A\longrightarrow A$ be an injective endomorphism such that $\beta(1)\neq 1$ and $\beta(A)$ is a hereditary sub-$C^*$-algebra and let $A\times_\beta \N=C^*(A,V)$. If does not exist any proper ideal $I$ of $A$ such that $\beta(I)\subseteq I$, then for every full  projection  $p\in A$ there exist a partial isometry $u\in A$ and $m\in \N$ such that $(V^*)^mu^*puV^m=(V^*)^mV^m=1$.
\end{lem}
\begin{proof} We claim that there exists $m\in \N$ such that $V^m(V^m)^*\lesssim p$. Observe that if $A$ is purely infinite then $p$ is a properly infinite full projection. So, we have that $VV^*\in \overline{ApA}=A$. Hence, $VV^*\lesssim p$, so that $m=1$ holds. 

Now suppose that $A$ has strict comparison. Then $T(A)$ is non-empty and compact. 
So, given any $k\in \N$ we set
$$\alpha=\text{inf }\{\tau(p):\tau\in T(A)\}\qquad \text{and}\qquad \gamma_k=\text{sup }\{\tau(V^k(V^*)^k):\tau\in T(A)\}\,.$$
Observe that, since $p$ is full, we have that $\alpha> 0$. Now, we claim that there exists $n\in\N$ such that $\gamma_n<1$. Indeed, it is enough to prove that there exists $n\in \N$ such that $1-V^n(V^*)^n$ is a full projection. Let us construct the ideal  
$$I:=\overline{\text{span }}\{x(V^l(V^*)^l-V^{l+1}(V^*)^{l+1})y: l\geq 0,\, x,y\in A\}\neq 0\,.$$
It is clear that $\beta(I)\subseteq I$. Therefore, by Lemma  \ref{full_element}, there exists $n\in \N$ such that 
$$(1-VV^*)+\cdots+\beta^{n-1}(1-VV^*)=(1-VV^*)+\cdots+(V^{n-1}(V^*)^{n-1}-V^{n}(V^*)^{n})=1-V^{n}(V^*)^{n}\,,$$
is a full projection. Therefore $\gamma_n<1$. By the same argument as in the proof of \cite[Lemma 3.2]{Ror}, we have that $\tau(V^{nl}(V^*)^{nl})\leq \gamma_n^l$ for every $l\in\N$. Then, there exists $l\in\N$ such that $\tau(V^{nl}(V^*)^{nl})\leq \gamma_n^l<\alpha\leq \tau(p)$. Since $A$ has strict comparison, we have that $V^{nl}(V^*)^{nl}\lesssim p$. So, there exists a partial isometry $u\in A$ such that $u^*u=V^{nl}(V^*)^{nl}$ and $uu^*\leq p$. Therefore $(V^*)^{nl}u^*puV^{nl}=(V^*)^{nl}(V^{nl}(V^*)^{nl})V^{nl}=1$, so we are done.
\end{proof}

\begin{lem} Let $A$ be a $C^*$-algebra of real rank zero, and let $\beta:A\rightarrow A$ be an extendible injective endomorphism with $\beta(A)$ being hereditary such that $\mathbb{T}(\beta)=\mathbb{T}$. Then, given any $a\in A^\sim$ and any $B$ hereditary sub-$C^*$-algebra of $A$ we have that
$$\text{inf }\{\|pa\beta(p)\|: p \text{ is a non-zero projection of }B\}=0\,.$$
\end{lem}
\begin{proof}
Let $a\in A^+$ and let $B$ be a hereditary sub-$C^*$-algebra of $A$. Given $\varepsilon>0$, by Theorem \ref{theor_simple} there exists $x\in B_+$ with $\|x\|=1$ such that $\|xa\beta(x)\|<\varepsilon/2$. Given $\delta>0$,  let $f_\delta:[0,1]\longrightarrow [0,1]$ be such that $f(t)=1$ for every $t\in [1-\delta/2,1]$ and such that $|f_\delta(t)-t|< \delta$ for every $0\leq t\leq 1$. Take $\delta>0$ such that  $\|f_\delta(x)a\beta(f_\delta(x))\|< \varepsilon$. Let $C=\{y\in B: f_\delta(x)y=yf_\delta(x)=y\}\neq 0$. Notice that $C$ is a hereditary sub-$C^*$-algebra of $B$. Since $C$ has real rank zero, there exists a non-zero projection $p\in C$, and by construction $pf_\delta(x)=f_\delta(x)p=p$. Therefore
$$\|pa\beta(p)\|=\|pf_\delta(x)a\beta(f_\delta(x)p)\|\leq \|f_\delta(x)a\beta(f_\delta(x))\|< \varepsilon\,.$$
\end{proof}

\begin{corol}\label{proj_aprox} Let $A$ be a $C^*$-algebra of real rank zero, and let $\beta:A\longrightarrow A$ be an extendible injective endomorphism with $\beta(A)$ being hereditary such that $\mathbb{T}(\beta^n)=\mathbb{T}$ for every $n>0$. Then, given any $\varepsilon>0$, $a_1,\ldots,a_k\in A^\sim$ and $n_1,\ldots,n_k\in \N$ and a projection  $p\in A$, there exists a projection $q\in pAp$ such that
$$\|qa_i\beta^{n_i}(q)\|<\varepsilon\qquad \text{ for every }i\in\{1,\ldots,k\}\,.$$
\end{corol}

A $C^*$-algebra $A$ is said to be \emph{weakly divisible} if given any projection $p\in A$, there exists a unital $*$-homomorphism $M_2\oplus M_3\longrightarrow pAp$ \cite[Lemma 5.2]{RP}. Conditions for a non-type $I$ real rank zero $C^*$-algebra being weakly divisible  are given in \cite[Theorem 5.8]{RP}. In particular, every simple non-type $I$ $C^*$-algebra of real rank zero is weakly divisible.  Observe that, if $A$ is weakly divisible or purely infinite, then the following statement holds:
\begin{align*} 
(\dagger)  \qquad &  \text{Given any } n\in\N \text{ and }p\in A \text{ there exists }  p_1,\ldots,p_n\in A \text{ non-zero pairwise orthogonal} \\
& \text{subprojections of }p \text{ with }p\in \overline{Ap_iA}\text{ for all } i 
\end{align*}  

\begin{prop}\label{full_proj_filling} Let $A$ be a unital $C^*$-algebra of real rank zero satisfying $(\dagger)$, let $\beta:A\longrightarrow A$ be an injective endomorphism such that $\beta(A)$ is a hereditary sub-$C^*$-algebra of $A$, and let  $A\times_{\beta}\N=C^*(A,V)$. If does not exist  any proper ideal $I$ of $A$ such that $\beta(I)\subseteq I$, then given any non-zero projection $p\in A$ there exist a full projection $q\in A$ and $c\in A\times_\beta\N$ such that $q= cpc^*$. 
\end{prop}
\begin{proof}
By Lemma \ref{full_element} there exists $n\in \N$ such that $p+\beta(p)+\cdots+\beta^n(p)$ is a full positive element of $A$. Since $A$ satisfies $(\dagger)$ there exist non-zero orthogonal projections $p_0,\ldots,p_n\in A$ such that $p_0+\cdots+p_n\leq p$ with $p\in \overline{Ap_iA}$ for all  $i\in\{0,\ldots,n\}$.  Observe that $p+\beta(p)+\cdots+\beta^n(p)$ lies in the ideal generated by $q':=p_0+\beta(p_1)+\cdots+\beta^n(p_n)$, so $q'$ is also a full positive element of $A$.  Denote  $p_i':=\beta^{i}(p_i)$ for every $i\in\{0,\ldots,n\}$.
Now we are going to use induction on $n$ to construct a projection $q\in A$ such that $p_0'+\cdots+p_n'\in \overline{AqA}$. 
The case $n=0$ is clear. Now, suppose that there exists a projection $q_{k-1}$ such that $p_0'+\cdots+p_{k-1}'\in \overline{Aq_{k-1}A}$.

Using the Riesz decomposition of $V(A)$ \cite{Z1} we have $p'_k\sim a_k\oplus b_k$ such that $a_k\lesssim q_{k-1}$ and $b_k\lesssim 1-q_{k-1}$. Let $v_k$ be the partial isometry such that $v_k^*v_k\leq p'_{k}$ and $v_kv_k^*\leq 1-q_{k-1}$. If we define the projection $q_k:=q_{k-1}+v_kv_k^*$, then we have that $p'_1+\cdots+p_k'\in \overline{Aq_kA}$. Therefore the projection $q:=q_n$ is full. If we define $c:=p_0+v_1V^{1}p_1+\cdots+v_nV^{n}p_n$, then we have that $$cpc^*=cc^*=p_0+v_1\beta(p_1)v_1^*+\cdots + v_n\beta^{n}(p_n)v^*_n =q\,,$$ as desired.
\end{proof}

\begin{theor}\label{theor_purely_inf} 
Let $A$ be a unital $C^*$-algebra of real rank zero satisfying $(\dagger)$ that has  strict comparison, let $\beta:A\longrightarrow A$ be an injective endomorphism such that $\beta(1)\neq 1$ and $\beta(A)$ is a hereditary sub-$C^*$-algebra of $A$. If $A\times_{\beta}\N$ is simple and $\beta(1)$ is a full projection of $A$, then $A\times_{\beta}\N$ is purely infinite simple $C^*$-algebra. 
\end{theor}
\begin{proof} It is enough to prove that given a positive element $x\in A\times_{\beta}\N$ there exist $a,b\in A\times_{\beta}\N$ such that $axb=1$. Let $E:A\times_{\beta}\N\longrightarrow A$ be the canonical faithful conditional expectation. So, $0\neq E(x)=c\in A_+$. Then, for  $\|c\|>\varepsilon>0$ we have that the hereditary sub-$C^*$-algebra $\overline{(c-\varepsilon)_+A(c-\varepsilon)_+}\subseteq c^{1/2}Ac^{1/2}$ has real rank zero. Hence, there exists a non-zero projection $p=c^{1/2}yc^{1/2}\in c^{1/2}Ac^{1/2}$. Then, $q=y^{1/2}cy^{1/2}$ is a projection, and $E(y^{1/2}xy^{1/2})=y^{1/2}cy^{1/2}=q$. Thus, we can assume that $E(x)=q$ is a non-zero projection. Given $1/2>\varepsilon>0$, there exists $x'=(V^*)^m d_{-m}+\cdots+q+\cdots+d_mV^m$, with $d_j\in A_+$ for every $j$, such that $\|x-x'\|<\varepsilon$.  By Corollary \ref{theor_simple_corol}, Theorem \ref{theor_simple} and Corollary \ref{proj_aprox}, there exists a non-zero projection $p\in qAq$ such that 
$$\|pd_i\beta^i(p) \|<\varepsilon/2m \qquad \text{and}\qquad \|\beta^i(p)d_{-i}p\|<\varepsilon/2m$$
for every $i\in \{1,\ldots,m\}$. Therefore 
$$\|pxp-p\|\leq \|pxp-px'p\|+\|px'p-p\|\leq \varepsilon + \varepsilon<1\,.$$
Then, $pxp$ is invertible in $p(A\times_\beta \N)p$, whence there exists $y\in p(A\times_\beta \N)p$ such that $ypxp=p$.  Since we are assuming that $A\times_\beta \N$ is simple and $\beta(1)$ is a full projection, \cite[Theorem 4.1]{S} implies that there are no non-trivial ideals $I$ of $A$ such that $\beta(I)\subseteq I$. Thus, by Proposition \ref{full_proj_filling}, there exist $c\in A\times_\beta \N$ and a full projection $q\in A$ such that $cpc^*= q$. 

By Lemma \ref{comparison}, there exist $m\in \N$ and a partial isometry $u\in A$ such that $(V^*)^mu^*quV^m=1$ and therefore 
$$(V^*)^mu^* (cypxpc^*) uV^m=(V^*)^mu^* cpc^* uV^m= (V^*)^mu^* q uV^m=1\,.$$
Thus, if we set $a:=(V^*)^mu^* cyp$ and $b:=pc^*uV^m$ we have $axb=1$, as desired.
\end{proof}

When $A$ is a purely infinite $C^*$-algebra, we generalize the result  of \cite{JKO}.

\begin{corol}\label{theor_pi_2} 
Let $A$ be a unital purely infinite $C^*$-algebra of real rank zero, let $\beta:A\longrightarrow A$ be an injective endomorphism such that $\beta(1)\neq 1$ is a full projection and $\beta(A)$ is a hereditary sub-$C^*$-algebra of $A$. Then  $A\times_{\beta}\N$ is a simple purely infinite $C^*$-algebra if and only if  $A\times_{\beta}\N$ is  simple.
\end{corol}
\begin{proof}  The proof works in the same way as that of Theorem \ref{theor_purely_inf}, but reminding that Lemma \ref{comparison} and condition $(\dagger)$ are also satisfied for purely infinite $C^*$-algebras.
\end{proof}

Finally, we can use Corollary \ref{theor_pi_2}  to characterize when a crossed product by an automorphism $A\times_\alpha \Z$ is simple and purely infinite.

\begin{corol}[{cf. \cite[Theorem 3.1]{JKO}}]\label{pi_auto} 
Let $A$ be a unital purely infinite $C^*$-algebra of real rank zero, and let $\alpha:A\longrightarrow A$ be an automorphism. Then $A\times_\alpha \Z$ is a simple purely infinite $C^*$-algebra if and only if $A\times_\alpha \Z$ is simple. 
\end{corol}
\begin{proof} The proof is a verbatim of the proof of \cite[Theorem 3.1]{JKO}. We only have to prove that there exist projections $p,e\in A$ and partial isometries $t,s\in A$ such that 
$$s^*s=\alpha(p)\,,\qquad ss^*=e<p\,,\qquad t^*t=1-\alpha(p)\qquad \text{ and }\qquad tt^*=1-e\,.$$
Indeed, since $1$ is a properly infinite projection, there exist mutually orthogonal projections $p_1,p_2,p_3\in A$, all them Murray-von Neumann equivalent to $1$. Observe that $\alpha(p_i)$ are mutually orthogonal  full properly infinite projection of $A$. Then, we have that $\alpha(p_1)\sim e <\alpha(p_2)$ for some projection $e\in A$. Since  $\alpha(p_3)\perp \alpha(p_1)$ and $e\perp \alpha(p_1)$,  by \cite[Proposition 2.5]{BRR} we have that $\alpha(p_1)$ and $e$ are homotopic equivalent, and hence $1-\alpha(p_1)$ and $1-e$ are homotopic equivalent, thus Murray-von Neumann too. Thus, setting $p:=\alpha(p_1)$ we have proved the claim.  

By the proof of \cite[Theorem 3.1]{JKO}, the dynamical system $(A,\alpha)$ is exterior equivalent to  $(A,\rho)$, where $\rho$ is the automorphism defined by $\rho(x)=(s+t)\alpha(x)(s+t)^*$ for every $x\in A$. So, it is enough to prove that $A\times_\rho\Z$ is simple and purely infinite. Notice that $\mathbb{T}(\alpha)=\mathbb{T}(\rho)$, and that $A$ is $\rho$-simple since it is $\alpha$-simple. Hence, $A\times_\rho \Z$ is a simple $C^*$-algebra. Then $p(A\times_\rho\Z)p\cong pAp\times_\rho\N$  is a full simple hereditary sub-$C^*$-algebra. Now, we have that  $pAp$ is a purely infinite $C^*$-algebra of real rank zero, and by construction $\rho(p)=s\alpha(p)s^*$ is a full projection of $pAp$. Thus, by  Theorem \ref{theor_pi_2} we have that $pAp\times_\rho\N$ is a purely infinite $C^*$-algebra, whence so is $A\times_\alpha \Z$. 
\end{proof}

\begin{exem} This is a generalization of Example \ref{exem1}(3) and Cuntz's construction of the algebras $\mathcal{O}_n$ \cite{Cu}. Let $\mathcal{U}_m$ be the $m$-infinity UHF algebra $\bigotimes_{n=1}^\infty M_m$, and let $B=\mathcal{U}_m\oplus\cdots\oplus \mathcal{U}_m$ be the direct sum of $n$ copies of $\mathcal{U}_m$, that is a nuclear unital weakly divisible $C^*$-algebra of real rank zero that absorbs $\mathcal{Z}$ and hence has strict comparison. Let us consider the endomorphism $\beta:B\longrightarrow B$ given by $\beta(x_1,\dots,x_n)=(P_1\otimes x_2,P_2\otimes x_3\cdots,P_n\otimes x_1)$ for every $(x_1,\dots,x_n)\in B$, where $P_1,\cdots,P_n\in M_m$ are  rank $1$ projections. Hence, $\beta$ is injective.  Observe that $\beta(1)\neq 1$ is a full projection of $B$.  It is clear that $B$ is $\beta$-simple and $\beta^k$ is outer for any $k>0$, since $B$ is a unital finite $C^*$-algebra. Hence, $B\times_\beta \N$ is simple by Theorem \ref{theor_simple}, and thus applying Theorem \ref{theor_pi} it is also a purely infinite $C^*$-algebra, in particular it is a Kirchberg algebra. Now, we use the modification of the Pimsner-Voiculescu six-term exact sequence given in \cite{Ror},
$$\xymatrix{K_0(B) \ar[r]^{1-\beta^*} & K_0(B)\ar[r]& K_0(B\times_\beta \N)\ar[d] \\ K_1(B\times_\beta \N) \ar[u] & K_1(B)\ar[l]& K_1(B)\ar[l]_{1-\beta^* } }$$
Notice that the induced map  $\beta^*:\Z[1/m]^n\longrightarrow \Z[1/m]^n$ is given by $$\beta^*(x_1,\ldots,x_n)=(x_2/m,\ldots,x_{n}/m,x_1/m)\,,$$ 
for every $(x_1,\ldots, x_n)\in \Z[1/m]^n$. Then, we can easily compute $ K_0(B\times_\beta \N)=\Z/(m^n-1)\Z$ and $ K_1(B\times_\beta \N)=0$. Thus, using the Kirchberg-Phillips classification theorems, we conclude that $B\times_\beta \N$ is stably isomorphic to the Cuntz algebra $\mathcal{O}_{m^n}$.

\end{exem}

\section*{Acknowledgments}

Part of this work was done during visits of the first author to the Department of Mathematics and Statistics of the University of Otago (New Zealand), and of the second author to the Institutt for Matematiske Fag, Norges Teknisk-Naturvitenskapelige Universitet (Norway). Both authors thank the host centers for their kind hospitality.

\end{document}